\documentclass[11pt]{amsart}
\usepackage[left=1in,right=1in,top=1in,bottom=1in]{geometry}
\usepackage{amssymb}
\usepackage{amsthm}
\usepackage{graphicx}
\newtheorem{theorem}{Theorem}[section]
\newtheorem{lemma}[theorem]{Lemma}
\newtheorem{corollary}[theorem]{Corollary}
\numberwithin{equation}{section}
\newtheorem{proposition}[theorem]{Proposition}
\newtheorem{example}[theorem]{Example}

\newcommand{\Z}{\mathbb{Z}}

\newcommand{\Q}{\mathbb{Q}}

\title{Homology cobordism and Seifert fibered $3$-manifolds}
\date{}
\author{Tim D. Cochran$^{\dag}$}
\address{Department of Mathematics MS-136, P.O. Box 1892, Rice University, Houston, TX 77251-1892}
\email{cochran@rice.edu}

\author{Daniel Tanner$^{\dag\dag}$}
\address{Epic Systems, 1979 Milky Way, Veroan, WI 53593}
\email{dtanner@epic.com}

\thanks{$^{\dag}$Partially supported by the National Science Foundation  DMS-1006908}
\thanks{$^{\dag\dag}$Partially supported by the National Science Foundation DMS-0739420}

\begin{document}
\date{\today}
\begin{abstract} It is known that every closed oriented $3$-manifold is homology cobordant to a hyperbolic $3$-manifold. By contrast we show that many homology cobordism classes contain no Seifert fibered $3$-manifold. This is accomplished by determining the isomorphism type of the rational cohomology ring of all Seifert fibered $3$-manifolds with no $2$-torsion in their first homology. Then we exhibit families of examples of $3$-manifolds (obtained by surgery on links), with fixed  linking form and cohomology ring, that  are not homology cobordant to any Seifert fibered space,  These are shown to represent distinct homology cobordism classes using higher Massey products and  Milnor's $\overline{\mu}$-invariants for links. 
\end{abstract}
\maketitle

\section{Introduction}\label{sec:intro}

Two closed oriented $3$-manifolds $M_1$ and $M_2$ are said to be \textit{homology cobordant} if there exists a compact oriented $4$-manifold $W$ such that $\partial W = M_1 \sqcup -M_2$ and the maps $H_n(M_i;\mathbb{Z})\rightarrow H_n(W;\mathbb{Z})$, $i=1,2$, induced by inclusion are isomorphisms for all $n$.  It is easily seen that homology cobordism gives an equivalence relation on $3$-manifolds. This can be considered in either the topological or smooth categories. The set of homology cobordism classes of $3$-manifolds forms a monoid with operation induced by connected sum and identity $[S^3]$. When restricted only to classes of homology $3$-spheres, this set forms an abelian group $\theta^H_3$. There are applications of homology cobordism of $3$-manifolds to the triangulation problem for topological manifolds of high dimension and to classical knot concordance (see ~\cite[Section 7.3]{Saveliev2002}).

In \cite{Li11} Livingston showed that every homology cobordism class contains an irreducible 3-manifold; that is, every $3$-manifold is homology cobordant to an irreducible $3$-manifold. In \cite{Myers1} Myers proved a stronger statement: every $3$-manifold is homology cobordant to a hyperbolic $3$-manifold. Thus every homology cobordism class has a hyperbolic representative. This is consonant with our expectation that hyperbolic $3$-manifolds are in some sense generic. Recall that a Seifert fibered $3$-manifold can be thought of as a circle bundle over a surface with some finite number of exceptional fibers. Here we address the question of which homology cobordism classes of $3$-manifolds contain a Seifert fibered $3$-manifold.  We use elementary arguments to show that the cohomology rings of Seifert fibered spaces are quite special and greatly restrict which homology cobordism classes contain a Seifert fibered space. Specifically, we show that:

\begin{theorem}\label{Main1}
Suppose that $M$ is a closed, connected, oriented 3-manifold such that $H_1(M)$ has no 2-torsion. If $M$ is homology cobordant to a Seifert fibered manifold then
\begin{itemize}
\item if $\beta_1(M)$ is odd, the rational cohomology ring of $M$ is isomorphic to that of $S^1\times \Sigma$ where $\Sigma$ is the closed orientable surface of genus $\frac{\beta_1(M)-1}{2}$; while
\item if $\beta_1(M)$ is even, the rational cohomology ring of $M$ is isomorphic to that of $\#_{\beta_1(M)} (S^1\times S^2)$.
\end{itemize}
\end{theorem}
\begin{corollary}\label{cor:main} Under the hypotheses of Theorem~\ref{Main1}, if $\beta_1(M)$ is odd and at least $3$ then for any non-zero $\alpha\in H^1(M;\Q)$ there exists $\beta\in H^1(M;\Q)$ such that $\alpha\cup \beta\neq 0$. If $\beta_1(M)$ is even then, for any $\alpha,\beta \in H^1(M;\Q)$, $\alpha\cup \beta=0$.
\end{corollary}

The above theorems do not tell us anything about homology cobordism classes of rational homology spheres. In this case, we make the following observation as a corollary to a previous calculation of the $\mathbb{Z}_p$ cohomology rings of Seifert fibered spaces \cite{BrydenZ2}.

\begin{theorem}\label{thm:main3}
Suppose that $M$ is a rational homology 3-sphere that is homology cobordant to a Seifert fibered space. Then, for each odd prime $p$, the cup product map $H^1(M;\mathbb{Z}_p)\times H^1(M;\mathbb{Z}_p)\to H^2(M;\mathbb{Z}_p)$ is trivial.
\end{theorem}

In Section~\ref{sec:examples} we use these results, and related notions of Milnor's $\overline{\mu}$ invariants for links and higher order Massey products, to exhibit infinite families of distinct homology cobordism classes, with fixed linking form and cohomology ring, that contain no Seifert-fibered $3$-manifold. Here are some sample results:

\newtheorem*{prop:betaodd}{Proposition~\ref{prop:betaodd}}
\begin{prop:betaodd} Suppose $m$ is a positive integer, $T$ is an abelian group of odd order and $\mathcal{L}$ is a non-singular linking form  on $T$. Then there exist an infinite number of homology cobordism classes of oriented $3$-manifolds, each with first homology $H\cong \Z^{2m+1}\oplus T$, with linking form $\mathcal{L}$ and having the cohomology ring of $\#_{2m+1}S^1\times S^2$, none of which contains a Seifert fibered manifold. 
\end{prop:betaodd}

\newtheorem*{prop:betaeven2}{Proposition~\ref{prop:betaeven2}}
\begin{prop:betaeven2} Suppose $M$ is the $3$-manifold obtained by zero framed surgery on a $2r$-component link $L$ in $S^3$ (with pairwise linking numbers zero) for which some Milnor's invariant $\overline{\mu}(ijk)$ is non-zero, then the homology cobordism class of $M$ contains no Seifert fibered manifold.
\end{prop:betaeven2}

\newtheorem*{prop:ratsphere}{Proposition~\ref{prop:ratsphere}}
\begin{prop:ratsphere} Suppose $M$ is the rational homology $3$-sphere obtained as $p$-surgery on each component of the Borromean rings where $p$ is an odd integer. Then the homology cobordism class of $M$ contains no Seifert fibered manifold.
\end{prop:ratsphere}

The much more difficult question of whether or not every integral homology $3$-sphere is smoothly homology cobordant to a Seifert fibered homology $3$-sphere seems to be open.

\section{Basic invariants of homology cobordism}\label{sec:invtshomcob}

We review some classical invariants of oriented homology cobordism, namely the isomorphism types of the homology itself, the cohomology ring, the triple cup product form (all with arbitrary abelian coefficients), and the $\Q/\Z$-linking form on the torsion subgroup of $H_1$. Other classical invariants of homology cobordism are afforded by Massey products and the Lie algebra associated to the lower central series of the fundamental group. Since these invariants obstruct both topological and smooth homology cobordism, henceforth we will not mention a specific category. Our results are the same in either category.

If $M_0$ and $M_1$ are closed oriented connected $3$-manifolds that are (oriented) homology cobordant via the $4$-manifold $W$ then there is an isomorphism between their cohomology rings with any coefficient ring $A$. Specifically if $j_0$ and $j_1$ are the inclusion maps, $M_i\to W$, then the maps $j_1^*\circ \left(j_0^*\right)^{-1}$ give isomorphisms $\phi^i$  on $H^i(M_0,A)$. Since the maps $j_i^*$ induce homomorphisms between cohomology rings, the maps $\phi^i$ preserve cup-products in the sense that
$$
\phi^1(x)\cup\phi^1(y)=\phi^2(x\cup y)
$$
for all $x,y \in H^1(M_0;A)$. Thus an isomorphism of the (oriented) cohomology rings is equivalent to the existence of isomorphisms $\phi^1$ and $\phi^2$ that satisfy this property. A similar result holds for the $\Q/\Z$-linking forms, $\lambda_i$, $i=0,1$ on the torsion subgroup of $H_1(M_i;\Z)$.

This information can be organized differently using the triple cup product form and is sometimes revealing.

\begin{proposition}\label{prop:gerges} If $M_0$ is homology cobordant to $M_1$ then there exists an isomorphism $\phi_1:H_1(M_1;\Z)\to H_1(M_0;\Z)$ such that for all integers $n\geq 0$ the induced maps
$$
\phi^1_n:H^1(M_0;\Z_n)\to H^1(M_1;\Z_n)
$$
satisfy:
\begin{itemize}
\item [a)] $\langle \alpha\cup\beta\cup\gamma, [M_0]\rangle=\langle \phi^1_n(\alpha)\cup\phi^1_n(\beta)\cup\phi^1_n(\gamma),[M_1]\rangle$ where $\alpha,\beta,\gamma\in H^1(M_0;\Z_n)$ and $[M_i]$ denoted the fundamental class (that is to say $\phi_1$ induces an isomorphism between the triple cup product forms); and
\item [b)] if $\lambda_1(x,y)=\lambda_0(\phi_1(x),\phi_1(y))$ for all $x,y\in TH_1(M_1;\Z)$, that is to say $\phi_1$ induces an isomorphism between these forms.
\end{itemize}
\end{proposition}

It is interesting that in ~\cite[Theorem 3.1]{CGO} it was shown that the existence of an isomorphism $\phi_1$ satisfying $a)$ and $b$ above is \textbf{equivalent to} the condition that $\phi_1$ exists such that $M_0$ and $M_1$ are equal in the bordism group $\Omega_3(K(H_1(M_0),1))$. Another way of saying this is that $M_0$ and $M_1$ are oriented cobordant in such a way that the inclusion-induced maps on \textbf{first} homology are isomorphisms, \textit{if and only if} there exists a $\phi_1$ that satisfies $a)$ and $b)$.

To what extent can these classical invariants be used to show that there exist homology cobordism classes that contain no Seifert fibered manifold?  If $H_1$ contains no $2$-torsion then Bryden-DeLoup have shown that \textit{any} torsion linking pairing can be realized by a Seifert fibered rational homology sphere ~\cite{BrydenD}. Hillman extended their work to the $2$-torsion case and found that there do exist a few linking pairings on $2$-torsion groups that \textit{cannot} be realized by any Seifert fibered space ~\cite[p.476]{HillmanSFS}. It follows from Hillman's result that indeed there do exist homology cobordism classes that contain no Seifert fibered manifold. Hence the torsion linking form gives obstructions in a very narrow range of cases. In the next section we will show, by naive calculations, that the rational cohomology ring gives more severe restrictions. The cohomology rings with $\mathbb{Z}_{p^r}$ coefficients of Seifert fibered spaces  have been calculated \cite{Aaslepp, BrydenZ2}. The techniques used in these calculations are quite intricate (involving a diagonal approximation of an equivariant chain complex of the universal cover). Oddly the computations of the integral and rational cohomology rings of Seifert fibered manifolds seem not to have appeared. 

More generally, Massey products of one dimensional cohomology classes provide obstructions to homology cobordism that go beyond the cup product structure. Specifically, if all the Massey products, $\langle x_1,\dots,x_{k-1}\rangle$,  of classes in $H^1(M;A)$ of length less than $k$ vanish then all the Massey products of length $k$ are uniquely defined in $H^2(M;A)$. If $M_0$ is homology cobordant to $M_1$ then, using naturality properties of Massey products, there must exist isomorphisms $\phi^1, \phi^2$ just as above which induce a correspondence between these length $k$ Massey products. In particular the minimal $k\leq \infty$ such that there exists a non-vanishing product of length $k$, called the \textit{Massey degree} of $M$, is an invariant of homology cobordism. Closely related is the graded Lie algebra associated to $G=\pi_1(M)$, obtained as the direct sum of the successive quotients, $G_k/G_{k+1}$, of the terms of lower central series of $G$ ~\cite{MKS}. The isomorphism type of this graded algebra is an invariant of homology cobordism ~\cite[Lemma 3.1]{St}. In particular if $H_1(M)$ is torsion-free of rank $r$ and $F$ is a free group of rank $r$, then the maximal $k\leq \infty$ such that $F/F_k\cong G/G_k$, the \textit{Milnor degree} of $M$, is an invariant of homology cobordism. For such manifolds the Massey degree is equal to the Milnor degree ~\cite[Prop.6.8]{CGO}\cite[Cor.2.9]{CM3}. In this paper we do not characterize the Massey product structure nor the Lie algebra associated to Seifert fibered spaces. This would be an excellent future project. Instead we use these to show that the number of distinct homology cobordism classes that do \textit{not} contain Seifert fibered manifolds is quite large, since, for example, even among manifolds with trivial cup product structure, there are many distinct homology cobordism classes distinguished by their higher Massey product structure. Examples are given in Section~\ref{sec:examples}.

\section{Proofs of the main theorems}\label{sec:proofs}

Our preferred model for a orientable Seifert fibered space is as follows. Let $B$ be a (possibly non-orientable) surface and form $B'$ from $B$ by removing $k$ open disks. Label the boundary components $b_1, \ldots, b_k$. Let $M'\rightarrow B'$ be the unique circle bundle over $B'$ whose total space is orientable. Note that if $B$ is orientable then $M'=B'\times S^1$. A choice of section $s:B'\rightarrow M'$ and orientation on $M'$ allow us to unambiguously define a meridian $\mu_i$ and a longitude $\lambda_i$ for each torus boundary component $\mathbb{T}_1,\ldots , \mathbb{T}_k$, where $\mu_i$ corresponds to $s(b_i)$ and $\lambda_i$ corresponds to a regular fiber. Form $M$ by doing $\alpha_i/\beta_i$ Dehn filling on each boundary component, where the meridian of the attached torus is glued to a smoothly embedded loop belonging to homotopy class $\alpha_i[\mu_i]+\beta_i[\lambda_i]$ in $\pi_1(\mathbb{T}_i)$. We may assume that both $\alpha_i$ and $\beta_i$ are non-zero. We encode $M$ with notation $M\cong(\pm g\ |\ \alpha_1/\beta_1, \ldots, \alpha_k/\beta_k)$ where $+g$ means $B$ is homeomorphic to a connected sum of g tori and $-g$ means $B$ is homeomorphic to a connected sum of g projective planes. Any closed orientable manifold that is a Seifert fibered space can be described in this way.

Note that the above construction allows us to easily write down a presentation for the fundamental group of $M$; if $M\cong(+g\ |\ \alpha_1/\beta_1, \cdots, \alpha_k/\beta_k)$ then

\begin{center}$\pi_1(M)=\langle x_i, y_i, \mu_j, t\ |\ [x_i,t],\ [y_i,t],\ [\mu_j,t],\ \mu_j^{\alpha_j}t^{\beta_j},\ [x_1, y_1]\cdots [x_g,y_g]\mu_1\cdots\mu_k \rangle$
\end{center}
and if $M\cong(-g\ |\ \alpha_1/\beta_1, \cdots, \alpha_k/\beta_k)$, $g\geq 1$,  then
 
\begin{center}$\pi_1(M)=\langle x_i, \mu_j, t\ |\ x_i t x_i^{-1} t,\ [\mu_j,t],\ \mu_j^{\alpha_j}t^{\beta_j},\ x_1^2\cdots x_g^2 \mu_1\cdots \mu_k  \rangle$
\end{center}
where in both cases $i$ ranges over $\{1, \ldots, g\}$ and $j$ ranges over $\{1, \ldots, k\}$. In both presentations, the generator $t$ is carried by a regular fiber.

\begin{lemma}\label{Lem:twotorsion}
Let $M\cong(-g\ |\ \alpha_1/\beta_1, \ldots, \alpha_k/\beta_k)$, $g\geq 1$. Then $H_1(M)$ has an element of order 2.
\end{lemma}

\begin{proof}
$H_1(M)$ has $\mathbb{Z}$-module presentation 
\begin{center}$\langle x_1,\ldots, x_g, \mu_1,\ldots, \mu_k, t\ |\ 0=\alpha_j\mu_j+\beta_jt = 2t = 2\Sigma x_i + \Sigma \mu_j \rangle$ for $1\leq j\leq k$, $1\leq i \leq g$.\end{center}

Suppose, for the purpose of contradiction, that $H_1(M)$ has no 2-torsion. Then $t=0$ and so $\alpha_j\mu_j=0$. Since $H_1(M)$ has no 2 torsion we may also assume $2\nmid\alpha_j$. Let $m_j=order(\mu_j)$ and $m=lcm(m_j)$. Then $m$ must be odd and $0=m(2\Sigma x_i + \Sigma \mu_j)=2m\Sigma x_i$. Since $H_1(M)$ has no 2-torsion, $m\Sigma x_i=0$. This must be a $\mathbb{Z}$ linear combination of our other relations $\{2\Sigma x_i + \Sigma \mu_j, \alpha_j \mu_j \}$, which it cannot be, since $m$ is odd.
\end{proof}

\begin{lemma}\label{Lem:regularfiber}
Let $M\cong(+g\ |\ \alpha_1/\beta_1, \ldots, \alpha_k/\beta_k)$. Let $\gamma$ be a regular fiber of $M'\rightarrow B'$. If $[\gamma]$ is of finite order in $H_1(M)$ then $\beta_1(M)=2g$. If $[\gamma]$ is of infinite order in $H_1(M)$ then $\beta_1(M)=2g+1$
\end{lemma}

\begin{proof} Abelianizing $\pi_1(M)$ and tensoring with $\mathbb{Q}$ gives $\mathbb{Q}$-module presentation

\begin{center}
$H_1(M;\mathbb{Q})=\langle x_i, y_i, \mu_j, t\ |\ \mu_i=\frac{-\beta_i}{\alpha_i}t,\ \sum_{j=1}^{k}\mu_j=0 \rangle$
\end{center}
where $i$ ranges over $\{1, \ldots, g\}$ and $j$ ranges over $\{1, \ldots, k\}$. Thus over $\mathbb{Q}$ the generators $\mu_i$ are redundant and the assertion follows immediately.\end{proof}

\begin{proof}[Proof of Theorem 1.1]  By Section~\ref{sec:invtshomcob} it suffices to assume that $M$ is itself Seifert fibered. By Lemma \ref{Lem:twotorsion}, since $H_1(M)$ has no 2-torsion, $M$ has an orientable base $B$, so $M\cong(+g\ |\ \alpha_1/\beta_1, \dots, \alpha_k/\beta_k)$. 

First we treat the case that $\beta_1(M)$ is odd. Then Lemma \ref{Lem:regularfiber} implies that $\beta_1(M)=2g+1$, $B$ has genus $g$ and $t=[\gamma]$ is of infinite order for any regular fiber $\gamma$. Choose a symplectic basis $[a_i], [b_i]$ for $H_1(B)$; that is, $a_i$ and $b_i$ are embedded curves in $B$ such that $a_i$ and $b_i$ intersect in a single point and $a_i$ and $b_j$ do no intersect if $i\neq j$. Furthermore, this basis can be chosen such that the curves lie in $B'$. Again by Lemma~\ref{Lem:regularfiber}, $\{[a_i],[b_i],[\gamma]\}=\{x_i,y_i,t\}$  is a basis for $H_1(M;\Q)$.

We will describe a collection of embedded closed oriented surfaces $\{A_i,B_i,T\}$ that is a basis for $H_2(M;\Q)$. Recall that $M'$ is just a product bundle over $B'$, so there exist embedded tori $A_i=a_i\times S^1$ and $B_i=b_i\times S^1$ in $M'\subseteq M$ that intersect either trivially if $i\neq j$ or in a regular fiber $\gamma_i$ if $i=j$. Note that $A_i$ is intersection-dual to $b_i$ and $B_i$ is intersection-dual to $a_i$. We only briefly describe the final surface $T$. For more details we refer the reader to ~\cite[Section 2.1]{Hatcher} where the existence (which is well-known) of this so-called \textit{horizontal surface} is detailed. First take $\alpha_1...\alpha_k$ parallel copies of the $2$-sided surface $B'$ and note that
$$
\partial((\alpha_1...\alpha_k)B')=\coprod_{j=1}^k(\alpha_1...\alpha_k)\mu_j.
$$
There is a $2$-disk $D_j$ associated to the $j^{th}$ Dehn filing whose boundary has the homology type of $\alpha_j\mu_j-\beta_j\gamma_j$. Hence by adjoining numerous copies of these disks to $(\alpha_1...\alpha_j)B'$ we can effectively eliminate copies of $\mu_j$ in favor of copies of $\gamma_j$. We arrive at a new surface, $T^o$, whose boundary is
$$
\partial T^o=\coprod_{j=1}^k(\alpha_1...\hat{\alpha}_j...\alpha_k)\beta_j\gamma_j=t[ \coprod_{j=1}^k(\alpha_1...\hat{\alpha}_j...\alpha_k)\beta_j],
$$
where we have replaced each $\gamma_j$ by $t$. Since $t$ has infinite order in $H_1(M;\Z)$, the coefficient of $t$ in this expression must be equal to zero (this coefficient is $m$ times the generalized Euler number of the Seifert fibration).
Hence $T^o$ represents a $2$-cycle of $M$. With more care (primarily because the $\gamma_j$, for varying $j$, are pairwise isotopic but not identical) one may get a closed oriented embedded surface $T$ as above. A key feature of $T$ is that it has non-zero  homological intersection number $m=\alpha_1...\alpha_k$ with a generic regular fiber $t$ and it has zero homological intersection with $[a_i]$ and $[b_i]$. Thus $\{[B_i],[A_i],\frac{1}{m}[T]\}$ is a basis of $H_2(M;\Q)$  that is intersection-dual (up to signs) to the  the basis $\{x_i,y_i,t\}$ for $H_1(M;\Q)$. 

We will use the fact that ``cup product is dual to intersection''.  To be specific, in order to show that the cohomology rings are isomorphic it suffices to exhibit isomorphisms $\phi_i:H_i(M;\Q)\to H_i(S^1\times \Sigma;\Q)$ for $i=1,2$ which preserve the homological intersection pairings
\begin{equation}\label{eq:pairing1}
H_2(M;\Q)\times H_2(M;\Q)\overset{\bullet}{\longrightarrow} H_1(M;\Q),
\end{equation}
and
\begin{equation}\label{eq:pairing2}
H_1(M;\Q)\times H_2(M;\Q)\overset{\bullet}{\longrightarrow} H_0(M;\Q),
\end{equation}
defined by $\widehat{x\bullet y}\equiv \widehat{x}\cup \widehat{y} $ where  $\widehat{[\star]}$ denotes the poincar\'e dual of $[\star]$.   Moreover we can compute homological intersection pairings by actual intersection, by appealing to a fact about intersections of smoothly embedded submanifolds (see \cite{Bredon1997} or ~\cite{Hutchings}): If $A$ and $B$ are smoothly embedded submanifolds of $M$ with $A$ and $B$ intersecting transversely, then cup products can be computed by
$$
\widehat{[A]}\cup \widehat{[B]} = \widehat{[A\cap B]}
$$
which translates to simply
\begin{equation}\label{eq:bullet}
[A]\bullet [B]=[A\cap B].
\end{equation}

If $\tau$ represents a circle in $S^1\times \Sigma$ and $\{\alpha_i,\beta_i\}$ represents a symplectic basis of $H_1(\Sigma)$, then let $\phi_1$ be the obvious map sending  $[a_i]$ to $\alpha_i$, $[b_i]$ to $\beta_i$ and $t$ to $\tau$. If $\mathcal{A}_i, \mathcal{B}_i$ are represented by the tori $S^1\times \alpha_i$ and $S^1\times \beta_i$ then  let $\phi_2$ be the map sending $[A_i]$ to $\mathcal{A}_i$, $[B_i]$ to $\mathcal{B}_i$ and $\frac{1}{m}[T]$ to $[\Sigma]$. Clearly these induce isomorphisms on rational cohomology. Since we have chosen $\phi_2$ to send the intersection dual of $[a_i]$ to the intersection dual of $\phi_1([a_i])$, et cetera,  the pairing ~(\ref{eq:pairing2}) is automatically preserved. We need only check that the pairing ~(\ref{eq:pairing1}) is preserved, that is that
\begin{equation}\label{eq:checkpairing}
\phi_2(a)\bullet\phi_2(b)=\phi_1(a\bullet b)
\end{equation}
for all $a,b$ in our basis of $H_2(M;\Q)$. Using equation~(\ref{eq:bullet}) we can now compute the pairing ~(\ref{eq:pairing1}) for $M$ and $S^1\times \Sigma$. 
$$
[A_i]\bullet [B_j] = [A_i\cap B_j]=\delta_{ij}[t], ~~\text{and}~~[A_i]\bullet [A_j] = [A_i\cap A_j]=0=[B_i]\bullet [B_j] = [B_i\cap B_j],
$$
where the latter hold since the tori $A_i$ and $B_i$ have product neighborhoods. Similarly since $T$ is embedded and oriented it also has a product neighborhood so
$$
[T]\bullet [T]=0.
$$
Finally, since $T\cap A_i=mB'\cap A_i$ is $m$ copies of $a_i$ and  $T\cap B_i$ is $m$ copies of $b_i$,
$$
\frac{1}{m}[T]\bullet [A_i]=\frac{1}{m}[T\cap a_i]= \frac{1}{m}[ma_i]=[a_i], ~~\text{and}~~~\frac{1}{m}[T]\bullet [B_i]=[b_i].
$$
But clearly the same results hold for our bases for $H_*(S^1\times\Sigma;\Q)$. For example
$$
\phi_2\left(\frac{1}{m}[T]\right)\bullet\phi_2([A_i])=[\Sigma]\bullet [\mathcal{A}_i]=[\alpha_i]=\phi_1([a_i])=\phi_1\left(\frac{1}{m}[T]\bullet [A_i]\right).
$$

Now we treat the case that $\beta_1(M)$ is even. Then Lemma \ref{Lem:regularfiber} implies that $\beta_1(M)=2g$, $B$ has genus $g$ and $t=[\gamma]=0$ in $H_1(M;\Q)$. Choose a symplectic basis $[a_i], [b_i]$ for $H_1(B)$ as before.  By Lemma~\ref{Lem:regularfiber}, $\{[a_i],[b_i],[\gamma]\}=\{x_i,y_i\}$  is a basis for $H_1(M;\Q)$. As above we have a collection of embedded closed oriented tori, $\{A_i,B_i\}$, that is a basis for $H_2(M;\Q)$. Our calculation now yields (since $[t]=0$),
$$
[A_i]\bullet [B_j] = [A_i\cap B_j]=[A_i]\bullet [A_j] = [A_i\cap A_j]=0=[B_i]\bullet [B_j] = [B_i\cap B_j]=0.
$$
Suppose $\{x_i\}$ and $\{S_i\}$, $1\leq i\leq 2g$ be the obvious bases of rational first and second homology of $\#_{2g}(S^1\times S^2)$ where $x_i$ is represented by the $i^{th}$ circle and $S_i$ is represented by the $i^{th}$ $2$-sphere. Let $\phi_1$ take $a_i$ to $x_{2i-1}$ and $[b_i]$ to $x_{2i}$ for $1\leq i\leq g$; and let $\phi_2$ take the $[A_i]$ and $[B_i]$ to the $S_i$ in a similar fashion. Just as above this are isomorphisms that preserve the pairings ~(\ref{eq:pairing2}) and ~(\ref{eq:pairing1}).

This completes the proof of Theorem~\ref{Main1}.

\end{proof}

\begin{proof}[Proof of Theorem~\ref{thm:main3}]
By Section~\ref{sec:invtshomcob} and Proposition~\ref{prop:gerges} the vanishing of all one-dimensional cup products is an invariant of homology cobordism. Thus  we need only show that all cup products of elements of $H^1(M;\mathbb{Z}_p)$ vanish when $M$ is a Seifert fibered space with $\beta_1(M)=0$ and $p\neq 2$. A calculation of the mod $p$ cohomology rings of orientable Seifert fibered manifolds was completed in \cite[Theorems~1.1-1.6]{BrydenZ2}. Suppose $\beta_1(M)=0$. If the base space of $M$ is orientable then $g=0$ by our Lemma~\ref{Lem:regularfiber}. In this case the result now follows from Theorems 1.1 and 1.3 of ~\cite{BrydenZ2}. If the base space is non-orientable the result follows from Theorems 1.4 and 1.6 of ~\cite{BrydenZ2}\end{proof}

\section{Examples}\label{sec:examples}

In this section we exhibit many homology cobordism classes that contain no Seifert fibered $3$-manifold.

\begin{proposition}\label{prop:betaodd} Suppose $m$ is a positive integer, $T$ is an abelian group of odd order and $\mathcal{L}$ is a non-singular linking form  on $T$. Then there exist an infinite number of homology cobordism classes of oriented $3$-manifolds, each with first homology $H\cong \Z^{2m+1}\oplus T$, with linking form $\mathcal{L}$ and having the cohomology ring of $\#_{2m+1}S^1\times S^2$, none of which contains a Seifert fibered manifold. 
\end{proposition}
\begin{proof} Since  $T$ contains no $2$-torsion, $\mathcal{L}$ can be realized by a Seifert fibered rational homology sphere $Q$ ~\cite{BrydenD}. Thus $H_1(Q)\cong T$. For  $d\geq 3$ let $N_d$ be the $3$-manifold obtained by zero framed surgery on the $3$-component link $L_d$ in Figure~\ref{fig:Ld}. Then $H_1(N_d)\cong \Z^3$ and $N_d$ has its first nonzero Massey product occuring at length $d$ (with $\Z$ or $\Q$ coefficients)~\cite[p.414]{CM3}\cite[Sections6, 7.4]{C4}. In particular it has vanishing cup products (of one dimensional classes). Now let $M_d=Q\# N_d\#_{(2m-1)}(S^1\times S^2)$. Then $H_1(M_d)\cong H$
\begin{figure}[htbp]
\setlength{\unitlength}{1pt}
\begin{picture}(200,110)
\put(0,10){\includegraphics{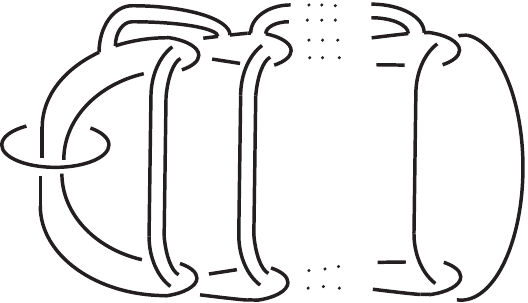}}
\put(60,-4){$d-1$ loops}
\end{picture}\label{fig:Ld}
\caption{$L_d$}
\end{figure}
and $M_d$ has linking form $\mathcal{L}$. Since each $M_d$ has vanishing cup products with rational coefficients, by Corollary~\ref{cor:main}, the homology cobordism class of $M_d$ contains no Seifert fibered manifold. But the homology cobordism classes represented by $M_d$ are distinct for varying $d$ since $M_d$ has rational Massey degree $d$ (see Section~\ref{sec:invtshomcob}).
\end{proof}

\begin{proposition}\label{prop:betaeven} Suppose $m\geq 2$ and $T$ is an abelian group of odd order. Then, for any non-singular linking form $\mathcal{L}$ on $T$ there exist an infinite number of homology cobordism classes of oriented $3$-manifolds, each with first homology $H\cong \Z^{2m}\oplus T$ and linking form $\mathcal{L}$, none of which contains a Seifert fibered manifold. 
\end{proposition}
\begin{proof} Since  $T$ contains no $2$-torsion, $\mathcal{L}$ can be realized by a Seifert fibered rational homology sphere $Q$ as above. Let $L_k$ be the $3$-component link obtained by taking a $(k,1)$ cable, $k\geq 1$, of one of the components of the Borromean rings. Let $N_k$ be the $3$-manifold obtained by zero framed surgery on $L_k$. Then $H^1(N_k)\cong \Z^3$ and the triple cup product of the obvious basis elements is $k[N_k]$.  Now let $M_k=Q\# N_k\#_{(2m-3)}(S^1\times S^2)$. Then $M_k$ has first homology $H$ and linking form $\mathcal{L}$. Moreover $N_k$ has non-zero cup products so by Corollary~\ref{cor:main}, no $M_k$ is homology cobordant to a Seifert fibered manifold. These homology cobordism classes are distinct because their triple cup product forms are not isomorphic ~\cite[Example 3.3]{CGO}. 
\end{proof}

Examples with fixed cohomology ring (fixed $k$) can be created by replacing $\#_{(2m-3)}(S^1\times S^2)$ by a manifold non-trivial higher-order Massey products as in Proposition~\ref{prop:betaodd}.

Any $3$-manifold $M$ with $\beta_1=r$ may be obtained as zero framed surgery on an $r$-component link $L$ (with null-homologous componentsa nd pairwise linking numbers zero) in a rational homology sphere $\Sigma$. There are explicit relationships between the Massey products of $M$ and the Massey products of $\Sigma-L$ and Milnor's $\overline{\mu}$-invariants of $L\hookrightarrow \Sigma$ (see for example ~\cite[Section 2]{CM3}). In particular the existence of cup-products is related to Milnor's $\overline{\mu}(ijk)$ of the $3$-component sublinks of $L$. For example if $L'$ is a $3$-component link (with linking numbers zero) in a homology sphere $\Sigma$ then the invariant $\overline{\mu}(123)$ is the algebraic number of triple points occuring when intersecting orientable bounding surfaces for the components. This is also equal to the degree of the map from $M_{L'}$ to the $3$-torus that is induced by the Hurewicz map $\pi_1(M_{L'})\to H_1(M_{L'})\cong \Z^3$. These arise from so-called ``Borromean rings-type interactions'' among the link components. In this language we have:

\begin{proposition}\label{prop:betaeven2} Suppose $M$ is the $3$-manifold obtained by zero framed surgery on a $2r$-component link $L$ in $S^3$ (with pairwise linking numbers zero) for which some Milnor's invariant $\overline{\mu}(ijk)$ is non-zero, then the homology cobordism class of $M$ contains no Seifert fibered manifold.
\end{proposition}
\begin{proof} Note $H_1(M;\Z)\cong \Z^{2r}$. If $M$ were homology cobordant to a Seifert fibered space then, by Corollary~\ref{cor:main},  for any $\alpha,\beta\in H^1(M;\Z)$, $\alpha\cup\beta$ is a torsion class in $H^2(M;\Z)$ and hence is zero. Thus the integral cohomology ring of $M$ is the same as that of a connected sum of $2r$ copies of $S^1\times S^2$. Then by ~\cite[Thm.6.10 k=2]{CGO} the Milnor's invariants of length $3$ vanish, contradicting our assumption.
\end{proof}
\begin{proposition}\label{prop:ratsphere} Suppose $M$ is the rational homology $3$-sphere obtained as $p$-surgery on each component of the Borromean rings where $p$ is an odd integer. Then the homology cobordism class of $M$ contains no Seifert fibered manifold.
\end{proposition}
\begin{proof} By Theorem~\ref{thm:main3} it suffices to show that there is a non-trivial cup product of $1$-dimensional classes with $\Z_p$-coefficients. This is true for $M$ because $\overline{\mu}(123)=\pm 1$ for the Borromean rings. A bounding surface for each component is a mod $p$ cycle and the fact that there is algebraically non-zero triple intersection between the $3$ bounding surfaces means that the triple cup product of their duals is non-trivial.
\end{proof}

\begin{example} Suppose that $M$ is zero framed surgery on each component of the $2$-component Whitehead link. Then $M$ is a Heisenberg manifold, the circle bundle over the torus with Euler class $\pm 1$. Hence $M$ admits a Seifert fibering. Yet $M$ has non-vanishing Massey products of length $3$ since $\overline{\mu}(1122)=\pm 1$ for the Whitehead link ~\cite[Thm3]{Koj1983}. Thus a Seifert fibered manifold with even $\beta_1$ may \textit{fail} to have the same rational Massey products as a connected sum of copies of $S^1\times S^2$. Thus the second part of Theorem~\ref{Main1} cannot be naively extended.
\end{example}

\bibliographystyle{plain}
\bibliography{mybib4_30_2012}

\begin{thebibliography}{10}

\bibitem{Aaslepp}
Kerstin Aaslepp, Michael Drawe, Claude Hayat-Legrand, Christian~A. Sczesny, and
  Heiner Zieschang.
\newblock On the cohomology of {S}eifert and graph manifolds.
\newblock In {\em Proceedings of the {P}acific {I}nstitute for the
  {M}athematical {S}ciences {W}orkshop ``{I}nvariants of {T}hree-{M}anifolds''
  ({C}algary, {AB}, 1999)}, volume 127, pages 3--32, 2003.

\bibitem{Bredon1997}
Glen~E. Bredon.
\newblock {\em Topology and geometry}, volume 139 of {\em Graduate Texts in
  Mathematics}.
\newblock Springer-Verlag, New York, 1997.
\newblock Corrected third printing of the 1993 original.

\bibitem{BrydenD}
J.~Bryden and F.~Deloup.
\newblock A linking form conjecture for 3-manifolds.
\newblock In {\em Advances in topological quantum field theory}, volume 179 of
  {\em NATO Sci. Ser. II Math. Phys. Chem.}, pages 253--265. Kluwer Acad.
  Publ., Dordrecht, 2004.

\bibitem{BrydenZ2}
J.~Bryden and P.~Zvengrowski.
\newblock The cohomology ring of the orientable {S}eifert manifolds. {II}.
\newblock In {\em Proceedings of the {P}acific {I}nstitute for the
  {M}athematical {S}ciences {W}orkshop ``{I}nvariants of {T}hree-{M}anifolds''
  ({C}algary, {AB}, 1999)}, volume 127, pages 213--257, 2003.

\bibitem{CM3}
Tim Cochran and Paul Melvin.
\newblock The {M}ilnor degree of a 3-manifold.
\newblock {\em J. Topol.}, 3(2):405--423, 2010.

\bibitem{C4}
Tim~D. Cochran.
\newblock Derivatives of links: {M}ilnor's concordance invariants and
  {M}assey's products.
\newblock {\em Mem. Amer. Math. Soc.}, 84(427):x+73, 1990.

\bibitem{CGO}
Tim~D. Cochran, Amir Gerges, and Kent Orr.
\newblock Dehn surgery equivalence relations on 3-manifolds.
\newblock {\em Math. Proc. Cambridge Philos. Soc.}, 131(1):97--127, 2001.

\bibitem{Hatcher}
Allen Hatcher.
\newblock Notes on basic $3$-manifold topology.
\newblock Preprint, http://www.math.cornell.edu/~hatcher/3M/3M.pdf.

\bibitem{HillmanSFS}
Jonathan~A. Hillman.
\newblock The linking pairings of orientable {S}eifert manifolds.
\newblock {\em Topology Appl.}, 158(3):468--478, 2011.

\bibitem{Hutchings}
Michael Hutchings.
\newblock Cup product and intersections.
\newblock 2011.
\newblock preprint available at
  http://www.osti.gov/eprints/topicpages/documents/record/635/1949160.html.

\bibitem{Koj1983}
Sadayoshi Kojima.
\newblock Milnor's {$\bar \mu $}-invariants, {M}assey products and {W}hitney's
  trick in {$4$} dimensions.
\newblock {\em Topology Appl.}, 16(1):43--60, 1983.

\bibitem{Li11}
Charles Livingston.
\newblock Homology cobordisms of {$3$}-manifolds, knot concordances, and prime
  knots.
\newblock {\em Pacific J. Math.}, 94(1):193--206, 1981.

\bibitem{MKS}
Wilhelm Magnus, Abraham Karrass, and Donald Solitar.
\newblock {\em Combinatorial group theory}.
\newblock Dover Publications Inc., Mineola, NY, second edition, 2004.
\newblock Presentations of groups in terms of generators and relations.

\bibitem{Myers1}
Robert Myers.
\newblock Homology cobordisms, link concordances, and hyperbolic
  {$3$}-manifolds.
\newblock {\em Trans. Amer. Math. Soc.}, 278(1):271--288, 1983.

\bibitem{Saveliev2002}
Nikolai Saveliev.
\newblock {\em Invariants for homology {$3$}-spheres}, volume 140 of {\em
  Encyclopaedia of Mathematical Sciences}.
\newblock Springer-Verlag, Berlin, 2002.
\newblock Low-Dimensional Topology, I.

\bibitem{St}
John Stallings.
\newblock Homology and central series of groups.
\newblock {\em J. Algebra}, 2:170--181, 1965.

\end{thebibliography}

\end{document}